\documentclass[11pt,reqno]{amsart} 
\usepackage{amssymb,amsthm,amsmath,epsfig,latexsym}
\usepackage{calc,times}
\usepackage{geometry} 
\begin{document}

\newcommand{\mmbox}[1]{\mbox{${#1}$}}
\newcommand{\proj}[1]{\mmbox{{\mathbb P}^{#1}}}
\newcommand{\Cr}{C^r(\Delta)}
\newcommand{\CR}{C^r(\hat\Delta)}
\newcommand{\affine}[1]{\mmbox{{\mathbb A}^{#1}}}
\newcommand{\Ann}[1]{\mmbox{{\rm Ann}({#1})}}
\newcommand{\caps}[3]{\mmbox{{#1}_{#2} \cap \ldots \cap {#1}_{#3}}}
\newcommand{\N}{{\mathbb N}}
\newcommand{\Z}{{\mathbb Z}}
\newcommand{\R}{{\mathbb R}}
\newcommand{\Tor}{\mathop{\rm Tor}\nolimits}
\newcommand{\Ext}{\mathop{\rm Ext}\nolimits}
\newcommand{\Hom}{\mathop{\rm Hom}\nolimits}
\newcommand{\im}{\mathop{\rm Im}\nolimits}
\newcommand{\rank}{\mathop{\rm rank}\nolimits}
\newcommand{\supp}{\mathop{\rm supp}\nolimits}
\newcommand{\arrow}[1]{\stackrel{#1}{\longrightarrow}}
\newcommand{\CB}{Cayley-Bacharach}
\newcommand{\coker}{\mathop{\rm coker}\nolimits}
\newcommand{\m}{{\frak m}}
\newcommand{\fitt}{{\rm Fitt}}
\newcommand{\A}{{\mathcal{A}}}
\newcommand{\K}{\mathbb K}
\newcommand{\OT}{R}  
\newcommand{\pan}{{\rm Span }}  
\newcommand{\M}{\mathsf M}
\newcommand{\Ima}{{\rm Im}\,}


\makeatletter
\renewcommand*\env@matrix[1][*\c@MaxMatrixCols c]{%
  \hskip -\arraycolsep
  \let\@ifnextchar\new@ifnextchar
  \array{#1}}
\makeatother


\sloppy
\newtheorem{defn0}{Definition}[section]
\newtheorem{prop0}[defn0]{Proposition}
\newtheorem{quest0}[defn0]{Question}
\newtheorem{thm0}[defn0]{Theorem}
\newtheorem{lem0}[defn0]{Lemma}
\newtheorem{corollary0}[defn0]{Corollary}
\newtheorem{example0}[defn0]{Example}
\newtheorem{remark0}[defn0]{Remark}
\newtheorem{prob0}[defn0]{Problem}

\newenvironment{defn}{\begin{defn0}}{\end{defn0}}
\newenvironment{prop}{\begin{prop0}}{\end{prop0}}
\newenvironment{quest}{\begin{quest0}}{\end{quest0}}
\newenvironment{thm}{\begin{thm0}}{\end{thm0}}
\newenvironment{lem}{\begin{lem0}}{\end{lem0}}
\newenvironment{cor}{\begin{corollary0}}{\end{corollary0}}
\newenvironment{exm}{\begin{example0}\rm}{\end{example0}}
\newenvironment{rem}{\begin{remark0}\rm}{\end{remark0}}
\newenvironment{prob}{\begin{prob0}\rm}{\end{prob0}}

\newcommand{\defref}[1]{Definition~\ref{#1}}
\newcommand{\propref}[1]{Proposition~\ref{#1}}
\newcommand{\thmref}[1]{Theorem~\ref{#1}}
\newcommand{\lemref}[1]{Lemma~\ref{#1}}
\newcommand{\corref}[1]{Corollary~\ref{#1}}
\newcommand{\exref}[1]{Example~\ref{#1}}
\newcommand{\secref}[1]{Section~\ref{#1}}
\newcommand{\remref}[1]{Remark~\ref{#1}}
\newcommand{\questref}[1]{Question~\ref{#1}}
\newcommand{\probref}[1]{Problem~\ref{#1}}

\newcommand{\std}{Gr\"{o}bner}
\newcommand{\jq}{J_{Q}}
\def\Ree#1{{\mathcal R}(#1)}

\numberwithin{equation}{subsection}  


\title{Star Configurations are Set-Theoretic Complete Intersections}
\author{\c{S}tefan O. Toh\v{a}neanu}

 \keywords{star configuration, set-theoretic complete intersection. \\ \indent Tohaneanu's Address: Department of Mathematics, University of Idaho, Moscow, Idaho 83844-1103, USA, Email: tohaneanu@uidaho.edu, Phone: 208-885-6234, Fax: 208-885-5843.}

\begin{abstract} Let $\A\subset\mathbb P^{k-1}$ be a rank $k$ arrangement of $n$ hyperplanes, with the property that any $k$ of the defining linear forms are linearly independent (i.e., $\A$ is called $k-$generic). We show that for any $j=0,\ldots,k-2$, the subspace arrangement with defining ideal generated by the $(n-j)-$fold products of the defining linear forms of $\A$ is a set-theoretic complete intersection, which is equivalent to saying that star configurations have this property.

\end{abstract}
\maketitle

\section{Introduction}

Let $R$ be a commutative unitary Noetherian ring and let $I$ be a proper ideal of $R$. Suppose $ht(I)=m$. Then $I$ is said to be {\em set-theoretic complete intersection} if there exist $f_1,\ldots,f_m\in I$ such that ${\rm rad}(I)={\rm rad}(\langle f_1,\ldots,f_m\rangle)$. For an ideal $J$, ${\rm rad}(J)$ denotes the radical of $J$. A projective scheme is set-theoretic complete intersection if its defining ideal has this property.

The problem of studying this property of schemes (varieties) is of great interest in algebraic geometry. This problem goes back to the 19-th century, when Kronecker and Cayley asked if any complex variety in $\mathbb P^3$ is set-theoretic complete intersection. Hartshorne in \cite{Ha} gives an counterexample for dimension 2 varieties: in $\mathbb C^4$ take the union of two $2-$dimensional planes with one point in common (or in $\mathbb P^3$ take the variety with defining ideal $\langle x_1,x_2\rangle\cap\langle x_3,x_4\rangle$); this is a connected variety which is not set-theoretically complete intersection. This leads to the famous Hartshorne Conjecture: every irreducible curve in $\mathbb P^3$ is set-theoretic complete intersection. At this moment it is not known if the smooth rational quartic in $\mathbb P^3$ (see \cite[Exercise 18.8]{Ei0}) is set-theoretic complete intersection or not.

There is a huge literature on the set-theoretic complete intersection property of ideals / varieties, and we will not present it here. The general idea is that it is difficult to show that set-theoretic complete intersection property is not satisfied (for his example Hartshorne makes use of the nonvanishing of the 3rd local cohomology module), and even more difficult to show that this property is satisfied (one has to exhibit the elements $f_1,\ldots,f_m$). The situation is the same if one restricts to subspace arrangements (as is the case of Hartshorne's example):  see the work of Schenzel and Vogel (\cite{ScVo1}), generalized by Lyubeznik in \cite{Ly}, to show that certain subspace arrangements are not set-theoretic complete intersections, or the methodology developed by Schmitt and Vogel in \cite{ScVo2} to show that in particular some arrangements of subspaces are set-theoretic complete intersections (see Example 4 in their paper).

\medskip

In these notes we show that star configurations are set-theoretic complete intersections.

Our base field is any field $\mathbb K$. Let $\A\subset\mathbb P^{k-1}$ be a rank $k$ hyperplane arrangement with $\ell_i\in R:=\mathbb K[x_1,\ldots,x_k],i=1,\ldots,n$ being the defining linear forms. For $2\leq s\leq k$, we will call $\A$ to be {\em $s-$generic} if and only if any $s$ of the defining linear forms are linearly independent.

Suppose that $\A$ is $k-$generic, and let $H_i=V(\ell_i),i=1,\ldots,n$. For $1\leq c\leq k-1$, consider {\em the star configuration $V_c$}:
$$V_c= \bigcup_{1\leq i_1<\cdots< i_c\leq n} H_{i_1}\cap \cdots \cap H_{i_c}.$$ This subspace arrangement has defining ideal $$I_{V_c}=\bigcap_{1\leq i_1<\cdots< i_c\leq n} \langle \ell_{i_1}, \ldots, \ell_{i_c} \rangle,$$ which is of height $c$.

If $c=1$, then $I_{V_1}=\langle \ell_1\cdots\ell_n\rangle$, which is a complete intersection ideal, so it is also a set-theoretic complete intersection. If $c=k-1$, then $V_{k-1}$ consists of a finite number of points (which are obviously rational points), and by \cite[Chapter V, Examples 3.13]{Ku}, this variety is set-theoretic complete intersection. The question is about what happens in-between.

In $R$ define the ideal $I(\A,a)=\langle\{\ell_{i_1}\cdots\ell_{i_a}|1\leq i_1<\cdots<i_a\leq n\}\rangle$. This is an ideal generated by {\em $a-$fold products of linear forms}. By Lemma \ref{lemma2} and from the classical theorem that the radical of an ideal is the intersection of its minimal primes, it becomes clear that $$I_{V_c}={\rm rad}(I(\A,n-c+1)), c=1,\ldots,k-1.$$ In fact, by using \cite[Proposition 2.9]{GHM}, one can show $I_{V_c}=I(\A,n-c+1)$, but this fact is not crucial for the main result.

We are going to show that $V_c$ are set-theoretic complete intersections, for any $c=1,\ldots,k-1$, by showing that the ideals $I(\A,n-c+1)$ have this property (Theorem \ref{main}).

\section{Main result}

If $I$ is an ideal in a commutative ring $R$, then the {\em radical of $I$} is by definition ${\rm rad}(I)=\{f\in R|f^n\in I\mbox{ for some }n\geq 1\}$.

For some $g_1,\ldots,g_s\in R$, by ${\rm rad}(g_1,\ldots,g_n)$ we denote the radical of the ideal generated by $g_1,\ldots,g_s$. Also, if $g\in R$ and $J$ is an ideal of $R$, by $\langle g, J\rangle$ we denote $\langle g\rangle+J$.

\medskip

We begin with a natural lemma.

\begin{lem}\label{lemma1} Let $R$ be a commutative unitary Noetherian ring. Let $I\subset R$ be an ideal, and let $f,g\in R$ be any two elements. Then, $${\rm rad}(fg,I)={\rm rad}(f,I)\cap {\rm rad}(g,I).$$
\end{lem}
\begin{proof} For any two ideals $J_1,J_2\subset R$, one has $${\rm rad}(J_1\cdot J_2)={\rm rad}(J_1\cap J_2)={\rm rad}(J_1)\cap{\rm rad}(J_2).$$ The last equality is well-known. The first equality comes from the fact that $J_1\cdot J_2\subseteq J_1\cap J_2$, and from the fact that if $f\in {\rm rad}(J_1\cap J_2)$, then there is a positive integer $m$ such that $f^m\in J_1$ and $f^m\in J_2$, leading to $f^{2m}\in J_1\cdot J_2$.

With this we have the following sequence of inclusions / equalities
$${\rm rad}(fg,I)\subseteq {\rm rad}(f,I)\cap {\rm rad}(g,I)={\rm rad}(\langle f, I\rangle\cdot\langle g, I\rangle)={\rm rad}(fg, fI,gI,I^2)\subseteq {\rm rad}(fg,I).$$ Hence the claim.
\end{proof}

\begin{lem}\label{lemma2} Let $\A\subset\mathbb P^{k-1}$ be a rank $k$ hyperplane arrangement with defining linear forms $\ell_1,\ldots,\ell_n\in R:=\mathbb K[x_1,\ldots,x_k]$. Let $j\in\{0,\ldots,k-2\}$.

Then, any minimal prime over $I(\A,n-j)$ is of the form $\langle \ell_{i_1},\ldots,\ell_{i_{j+1}}\rangle$, for some $1\leq i_1<\cdots< i_{j+1}\leq n$. Furthermore, $ht(I(\A,n-j))\leq j+1$, and in particular, if $\A$ is $k-$generic, then $ht(I(\A,n-j))=j+1$.
\end{lem}
\begin{proof} The proof follows the same argument from the beginning of Section 2 in \cite{To1}. The idea is the following: let $\frak p$ be a minimal prime over $I(\A,n-j)$. So $\ell_1\cdots\ell_{n-j}\in\frak p$, giving that $\ell_{i_0}\in\frak p$, for some $i_0\in\{1,\ldots,n-j\}$. Let's assume $i_0=1$. Next consider $\ell_2\cdots\ell_{n-j+1}\in\frak p$, and do the same obtaining $\ell_2\in\frak p$. Recursively, one obtains $\ell_1,\ell_2,\ldots,\ell_{j+1}\in\frak p$. But $I(\A,n-j)\subset \langle \ell_1,\ldots,\ell_{j+1}\rangle\subseteq \frak p$. Since $\frak p$ was taken to be a minimal prime, then the last inclusion is in fact equality. We mention here as well that in general $\ell_1,\ldots,\ell_{j+1}$ will \underline{not} minimally generate $\frak p$.

The remaining part of the statement is immediate from the definition of the height of an ideal.
\end{proof}

The radical of an ideal is the intersection of its minimal prime ideals (see \cite[Corollary 2.12]{Ei0}). With this important result in mind, coupled with Lemma \ref{lemma2}, we can begin proving the main results of this note.

Unless one talks about multiarrangements (and this is not the case here), any hyperplane arrangement is $2-$generic. We have the following proposition.

\begin{prop} \label{prop1} Let $\A\subset\mathbb P^{k-1}$ be a rank $k$ hyperplane arrangement with defining linear forms $\ell_1,\ldots,\ell_n\in R:=\mathbb K[x_1,\ldots,x_k]$, and with $\gcd(\ell_i,\ell_j)=1,i\neq j$. Then $${\rm rad}(I(\A,n-1))={\rm rad}(\ell_1(\ell_3\cdots\ell_n+\cdots+ \ell_2\cdots\ell_{n-1}),\ell_2\cdots\ell_n).$$
\end{prop}
\begin{proof} From Lemma \ref{lemma2}, all the minimal primes over $I(\A,n-1)$ are of the form $\langle \ell_i,\ell_j\rangle, i\neq j$, hence $${\rm rad}(I(\A,n-1))=\bigcap_{1\leq i<j\leq n}\langle \ell_i,\ell_j\rangle.$$

On the other hand, from Lemma \ref{lemma1}, one has $${\rm rad}(\ell_1(\ell_3\cdots\ell_n+\cdots+ \ell_2\cdots\ell_{n-1}),\ell_2\cdots\ell_n)= I_1\cap I_2,$$ where $I_1:={\rm rad}(\ell_1,\ell_2\cdots\ell_n)$ and $I_2:={\rm rad}(\ell_3\cdots\ell_n+\cdots+ \ell_2\cdots\ell_{n-1},\ell_2\cdots\ell_n)$.

Again from Lemma \ref{lemma1}, one has $$I_1=\bigcap_{a=2}^n\langle \ell_1,\ell_a\rangle,$$ and
\begin{eqnarray}
I_2&=&\bigcap_{b=2}^n{\rm rad}(\ell_3\cdots\ell_n+\cdots+ \ell_2\cdots\ell_{n-1},\ell_b)\nonumber\\
&=& \bigcap_{b=2}^n{\rm rad}(\ell_2\cdots\widehat{\ell_b}\cdots\ell_n,\ell_b)\nonumber\\
&=& \bigcap_{2\leq i\neq j\leq n}\langle \ell_i,\ell_j\rangle.\nonumber
\end{eqnarray} Hence the claim.
\end{proof}

\medskip

The main result of the notes is the following.

\begin{thm} \label{main} Let $\A\subset\mathbb P^{k-1}$ be a rank $k$ hyperplane arrangement with defining linear forms $\ell_1,\ldots,\ell_n\in R:=\mathbb K[x_1,\ldots,x_k]$. Suppose $\A$ is $k-$generic, and let $j\in\{1,\ldots,k-2\}$. Then
$${\rm rad}(I(\A,n-j))={\rm rad}(F_1,\ldots,F_j,\ell_{j+1}\cdots\ell_n),$$ where $\displaystyle F_u:=\ell_u\left(\sum_{I\subseteq\{u+1,\ldots,n\}, |I|=n-j-1}\ell_I\right)$, for $u=1,\ldots,j$, and $\ell_I:=\prod_{i\in I}\ell_i$.
\end{thm}
\begin{proof} We prove the theorem by induction on $j\geq 1$.

\medskip

\noindent {\bf Base Step:} $j=1$. This is the case of Proposition \ref{prop1}.

\medskip

\noindent {\bf Induction Step:} $j\geq 2$.

Let $\frak p$ be a minimal prime over $\langle F_1,\ldots,F_j,\ell_{j+1}\cdots\ell_n\rangle$. Then there exists $v\in\{j+1,\ldots,n\}$ such that $\ell_v\in \frak p$.

For each $u=1,\ldots,j$, one has $$F_u=\underbrace{\ell_u\left(\sum_{I\subseteq\{u+1,\ldots,n\}\setminus\{v\}, |I|=n-j-1}\ell_I\right)}_{G_u}+\ell_vQ,\mbox{ for some }Q\in R.$$ So $G_u\in\frak p,u=1,\ldots,j$.

For $u=j$, $|\{u+1,\ldots,n\}\setminus\{v\}|=n-j-1$, and so $G_j=\ell_j\cdots\widehat{\ell_v}\cdots\ell_n$.

Denote $\A_v=\A\setminus\{\ell_v\}$. We have $|\A_v|=n-1$, and $\A_v$ is $k-$generic if $rank(\A_v)=k$, or $\A_v$ is $(k-1)-$generic if $rank(\A_v)=k-1$. By induction, $${\rm rad}(I(\A_v,\underbrace{(n-1)-(j-1)}_{n-j}))={\rm rad}(G_1,\ldots,G_{j-1},\ell_j\cdots\widehat{\ell_v}\cdots\ell_n).$$ So $I(\A_v,n-j)\subset\frak p$, which, from Lemma \ref{lemma2}, leads to the existence of $i_1<\cdots<i_j$ in $\{1,\ldots,n\}\setminus\{v\}$ with $$I(\A_v,n-j)\subset\langle \ell_{i_1},\ldots,\ell_{i_j}\rangle\subset\frak p.$$

Since $\ell_v\in\frak p$, and since $\frak p$ was taken to be minimal prime, then $$I(\A,n-j)\subset \langle I(\A_v,n-j),\ell_v\rangle\subset \langle \ell_{i_1},\ldots,\ell_{i_j},\ell_v\rangle=\frak p.$$ Again from Lemma \ref{lemma2} one obtains that $\frak p$ is a minimal prime over $I(\A,n-j)$. Therefore, every minimal prime over $\langle F_1,\ldots,F_j,\ell_{j+1}\cdots\ell_n\rangle$ is also a minimal prime over $I(\A,n-j)$, giving that $${\rm rad}(F_1,\ldots,F_j,\ell_{j+1}\cdots\ell_n)\supseteq {\rm rad}(I(\A,n-j)).$$ The reverse inclusion that will complete the proof comes from observing that $\langle F_1,\ldots, F_j,\ell_{j+1}\cdots\ell_n\rangle\subset I(\A,n-j)$.
\end{proof}

\begin{rem} The {\em arithmetic rank} of an ideal $I$, denoted $ara(I)$, is the minimum number of elements in $R$ that generate $I$ up to its radical ideal. So $I$ is set-theoretic complete intersection if and only if $ara(I)=ht(I)$ (in general one has $\geq$ happening).

In \cite{Ly} lower bounds for $ara(I)$ are presented (under certain conditions), allowing one to decide if $I$ is not a set-theoretic complete intersection, whereas in \cite{ScVo2} upper bounds of this invariant are given (also under certain conditions), which can be used to show that $I$ is a set-theoretic complete intersection. \cite[Theorem 1]{ScVo2} is even more detailed by presenting a formula for $ara(I)$. One would be tempted to use this result to show that star configurations are set-theoretic complete intersection, but an immediate application (meaning taking $a_{ij}$ in their paper to be our $\ell_i$'s) does not work, as assumption (ii) is not satisfied.

Nevertheless, very useful for our discussion is \cite[Lemma]{ScVo2}, which we are presenting next: let $R$ be a commutative ring with non-zero identity. Let $P$ be a finite subset of elements of $R$. Let $P_0,\ldots, P_r$ be subsets of $P$ such that
\begin{enumerate}
  \item[(i)] $\displaystyle\bigcup_{l=0}^r P_l=P$;
  \item[(ii)] $P_0$ has exactly one element;
  \item[(iii)] if $p$ and $p''$ are different elements of $P_l (0<l\leq r)$, there is an integer $l'$ with $0\leq l'<l$, and an element $p'\in P_{l'}$ such that $p'|p\cdot p''$.
\end{enumerate} Setting $\displaystyle q_l=\sum_{p\in P_l}p^{e(p)},$ where $e(p)\geq 1$ are arbitrary integers, then $${\rm rad}(P)={\rm rad}(q_0,\ldots,q_r).$$

This result provides an alternative proof of our Proposition \ref{prop1}. If we take $P_0=\{\ell_2\cdots\ell_n\}$, and $P_1=\{\ell_1\ell_3\cdots\ell_n,\ldots,\ell_1\ell_2\cdots \ell_{n-1}\}$, the assumptions (i), (ii) and (iii) above are satisfied, hence, by taking $e(p)=1$, one obtains $${\rm rad}(P_0\cup P_1)={\rm rad}(\ell_2\cdots\ell_n,\ell_1\ell_3\cdots\ell_n+\cdots+\ell_1\ell_2\cdots \ell_{n-1}).$$

Furthermore, this lemma helps prove that, for any hyperplane arrangement $\mathcal A$, $$ara(I(\A,n-j))\leq j+1.$$ One can see this by taking $P_0:=\{\ell_{j+1}\cdots\ell_n\}$, and for $u=1,\ldots,j$, $P_u:=\{\ell_{j-u+1}\ell_I|I\subset \{j-u+2,\ldots,n\},|I|=n-j-1\}$. Assumptions (i) and (ii) are immediately satisfied, and with a bit of work, also condition (iii) can be verified.

In general $ht(I(\A,n-j))\leq \min\{k-1,j+1\}, j=0,\ldots,n-d-1$, where $d$ is the minimum distance of a linear code built from $\A$ (see \cite{To1} for more details). In the case of $\A$ being $k-$generic, $d=n-k+1$, and furthermore, from Lemma \ref{lemma2} one has $ht(I(\A,n-j))=j+1$, this way obtaining an alternative proof of our Theorem \ref{main}.

We can generalize this to $\A$ being $s-$generic, $2\leq s\leq k$, obtaining that for any $j\in\{0,\ldots,s-2\}$ the ideals $I(\A,n-j)$ are set-theoretic complete intersections.
\end{rem}

\medskip

We end with the following question: for any hyperplane arrangement $\A$, are all ideals generated by $a-$fold products of linear forms set-theoretic complete intersections?

\medskip

\noindent {\bf Update (01/25/2016):} The answer to the above question is NO. Consider $\A\subset\mathbb P^3$ with defining linear forms $\{x,y,x+y,z,w,z+w\}$. We have $${\rm rad}(I(\A,4))=\langle x,y\rangle\cap\langle z,w\rangle,$$ and this is the case of Hartshorne's example. For all the other values of $a\in\{1,\ldots,6\}$, the ideal $I(\A,a)$ is set-theoretic complete intersection: (1) if $a=1,2$, we have $ht(I(\A,a))=4$ because the minimum distance of the associated linear code is $2$ (see, for example, \cite{To1}), hence s.t.c.i.; (2) if $a=3$, we have $ht(I(\A,a))=3$, hence the variety of this ideal consists of a finite set (six for this example) of rational points and by \cite[Chapter V, Examples 3.13]{Ku} this is s.t.c.i.; (3) if $a=5$, $I(\A,a)$ is s.t.c.i. by Proposition \ref{prop1}; (4) if $a=6$, then $I(\A,a)=\langle xy(x+y)zw(z+w)\rangle$ which is a complete intersection.

Furthermore, one can impose various combinatorial conditions on $\A$, such as $\A$ being supersolvable or being a graphic arrangement (see \cite{OrTe} for definitions and details), and ask the same question whether or not all ideals generated by $a-$fold products of linear forms defining $\A$ are set-theoretic complete intersections. The same example will also answer this question negatively: one can view this arrangement as the graphic arrangement associated to either the graph consisting of two disjoint triangles, or the graph consisting of two triangles with a vertex in common. In both cases make the change of variables assigning variables $x,y$ to two linear forms corresponding to two edges of one triangle, and variables $z,w$ to two linear forms corresponding to two edges of the other triangle. The theory says that $\A$ is a free arrangement which is equivalent to being supersolvable.

\vskip .1in

\noindent{\bf Acknowledgement} We are very grateful to the anonymous referee for all the corrections and suggestions.

\renewcommand{\baselinestretch}{1.0}
\small\normalsize 

\bibliographystyle{amsalpha}

\end{document}